\documentclass[
 a4paper, 
 reqno, 
 oneside, 
 12pt
 ]{amsart}
 
 \usepackage[utf8]{inputenc}

 \usepackage[dvipsnames]{xcolor} 
 \colorlet{cite}{red}
 \usepackage{tikz}  

 \usetikzlibrary{arrows,positioning}
 \tikzset{ 
 	baseline=-2.3pt,
 	text height=1.5ex, text depth=0.25ex,
 	>=stealth,
 	node distance=2cm,
 	mid/.style={fill=white,inner sep=2.5pt},
 }

 \usepackage{lmodern}
 \usepackage{tikz-cd}
 \usepackage{amsthm, amssymb, amsfonts}
 \usepackage{pb-diagram}
 
 \usepackage{amsthm, amssymb, amsfonts}	
 \usepackage[all]{xy}
 \usepackage{graphicx,caption,subcaption}
 \usepackage{braket}  
 \usepackage{microtype}
 \usepackage{DotArrow}
 \usepackage[makeroom]{cancel}
 \usepackage[%
 bookmarks=true,			
 unicode=true,			
 pdftitle={Cosymplectic},		%
 pdfauthor={Lopez-Martinez},	%
 pdfkeywords={Monodromy problem, tangential-center problem},	
 colorlinks=true,		
 linkcolor=black,			
 citecolor=black,		
 filecolor=magenta,		
 urlcolor=RoyalBlue			
 ]{hyperref}				
 \usepackage{cleveref}
\usepackage[pagewise]{lineno}

 \newtheoremstyle{mydef}
 {}		
 {}		
 {}		
 {}		
 {\scshape}	
 {. }		
 { }		
 {\thmname{#1}\thmnumber{ #2}\thmnote{ #3}}	
 
 \newtheorem{theorem}{Theorem}[section]
 \newtheorem*{theorem*}{Theorem}
 \newtheorem{proposition}[theorem]{Proposition}
 \newtheorem*{proposition*}{Proposition}
 \newtheorem{lemma}[theorem]{Lemma}
 \newtheorem*{lemma*}{Lemma}
 \newtheorem{corollary}[theorem]{Corollary}
 \newtheorem*{corollary*}{Corollary}
 \theoremstyle{definition}
 \newtheorem{definition}[theorem]{Definition}
 \newtheorem{example}[theorem]{Example}

 \newtheorem{remark}[theorem]{Remark}
 \theoremstyle{remark}

 \newtheorem*{conjecture*}{Conjecture}
 \usepackage{tikz}

 \newcommand{\R}{\mathbb R}
 \newcommand{\F}{\mathcal F}

 \newcommand{\xto}{\xrightarrow}
 \newcommand{\action}{\curvearrowright}
 
 \newcommand{\toto}{\rightrightarrows}

 \newcommand{\xfrom}{\xleftarrow}

 \newcommand{\rank}{\text{rank}}
 
 \newcommand{\gpd} {\mathcal{G}}
 \newcommand{\Gp}   {G}
 \newcommand{\g}   {\mathfrak{g}}
 \newcommand{\w}   {\omega}
 \newcommand{\Jmm} {\mathcal{J}}
 \newcommand{\Ker} {\mathrm{Ker}}

\author{Daniel L\'opez Garcia , and Nicolas Martinez Alba}
\subjclass{53D20, 58H05, 58D19, 37J15, 53D17}
 \address{}
 \date{\today}

 \title{Reduction of cosymplectic groupoids by cosymplectic moment maps}

\begin{document}
 	\maketitle
 	 	\begin{abstract}
 	 The Marsden-Weinstein-Meyer symplectic reduction has an analogous version  for cosymplectic manifolds. In this paper we extend this cosymplectic reduction to the context of groupoids.  Moreover, we prove how in the case of an algebroid associated to a cosymplectic groupoid, the integration commutes with the reduction (analogously to what happens in Poisson geometry). On the other hand, we show how the cosymplectic reduction of a groupoid induces a symplectic reduction on a canonical symplectic subgroupoid. Finally, we study what happens to the multiplicative Chern class associated with the $S^1$-central extensions of the reduced groupoid.\\

 	\noindent  \begin{tiny}KEYWORDS.\end{tiny} Group actions, cosymplectic groupoids, moment maps, co-symplectic reduction, Poisson manifolds.   
 	\end{abstract}
 	\vspace{1cm}

 	 	\noindent {\bf Acknowledgments:} We are grateful to David Iglesias for his valuable help in the discussions that inspired this work.  We thank Cristian Ortiz for helpful and insightful discussions on the reduction of Poisson groupoids. We would like to thank Javier de Lucas and Xavier Rivas for inviting us to their Gamma seminar, where many ideas from this article were improved. We are grateful to Luca Vitagliano for his detailed explanations on cosymplectic structures via line bundles and their associated cosymplectic groupoids.         We thank the anonymous referee for valuable comments and detailed suggestions that significantly improved the presentation of this article. L\'opez Garcia was partially supported by grant \# 2022/04705-8, S\~ao Paulo Research Foundation (FAPESP).

 	\section{Introduction}
 	Cosymplectic manifolds arise as an odd counterpart of symplectic manifolds \cite{Li}. They can also be found as the underlying structures of other types of geometry as CR geometry \cite{DT}, the odd-dimensional counterpart of K\"{a}hler manifolds and co-K\"{a}hler structures \cite{Bl1}, almost contact manifolds \cite{Bl0}, and in the description of time-dependent mechanics \cite{JLucas,GM} and classical field theories \cite{LSM}. Recently, these structures have been studied in the context of Lie groupoids giving rise to cosymplectic groupoids and their respective infinitesimal counterpart \cite{AW,FI} as central extensions of the underlying Poisson structure.

	In the literature, several results about symmetries and reductions for geometric structures exist.  In the case of cosymplectic geometry, we can consider two of them:  the reduction of the associated Poisson structure and the Albert reduction using cosymplectic moment maps (analogous to the Marsden-Weinstein-Meyer symplectic reduction).  The key difference between them is that the Poisson reduction does not preserve the nondegenerate condition in the cosymplectic case. On the other hand, the quotient of the level set of the cosymplectic moment map becomes cosymplectic again. For symplectic groupoids and Poisson structures on the units, these symmetries and reductions can be related by the integration theory of Lie groupoids \cite{FOR}.   The study of symmetries and reduction  was later extended to Lie algebroids equipped with infinitesimal multiplicative 2-forms \cite{BC, BC2}.  It is important to note that in \cite{CO}, the case of IM 2-form reductions for  quotients of vector bundles  has been studied; in particular, Poisson structure quotients are studied. In contrast, our results focus on moment map reductions targeting co-symplectic structure reductions.

	These notes extend the reduction technique in \cite{BC2} to the particular case of cosymplectic groupoids and their Lie algebroids, taking into account the results in \cite{FI} and adapting the cosymplectic reduction theorem as in \cite{Al,JLucas}. In Section \ref{sec:Preliminares} we start with an overview of cosymplectic manifolds, groupoids, and their infinitesimal data. 
	In Section \ref{sect: MainTheorem} we present the first main result of this article, which is a cosymplectic reduction by using a cosymplectic moment map in the context of Lie groupoids. 	Indeed, we consider symmetries of the cosymplectic forms equipped with a cosymplectic moment map that is also a groupoid morphism. Under certain conditions, the quotient of the 0-level set by the group action turns out to be a cosymplectic groupoid again (see Theorem \ref{thm: MainTheorem}).

Section \ref{sect: consecuencias} is intended to review some consequences of Theorem \ref{thm: MainTheorem}. In Section \ref{sect:consecuencias1} we consider the algebroid of a cosymplectic groupoid and some symmetries at the infinitesimal level to define an algebroid morphism. This map can be used to perform a reduction at the infinitesimal level. On the other hand, using Lie's second theorem we can integrate this algebroid morphism to a groupoid morphism which turns out to be a cosymplectic moment map. In Theorem \ref{thm:red vs int}, we show that the algebroid of the reduced cosymplectic groupoid corresponds to the reduction at the algebroid level.

Section \ref{sect:consecuencias2} investigates how the cosymplectic reduction of a cosymplectic groupoid is related to the symplectic reduction of a particular symplectic subgroupoid. On a cosymplectic groupoid, there is a symplectic foliation on the space of arrows. The leaf containing the units is a symplectic subgroupoid. In Theorem \ref{thm:mainSimplecticLeaf}, we show how cosymplectic reduction induces a Marsden-Weinstein symplectic reduction on the previously mentioned leaf. In section \ref{sect:consecuencias3}, using the relation between cosymplectic groupoids and $S^1$-central extensions with vanishing multiplicative Chern class, we study what happens to the multiplicative Chern class of the reduced groupoid. In Theorem \ref{thm:mainChern0}, we consider an oversymplectic groupoid with an action of a compact, connected Lie group; if the multiplicative Chern class of the initial groupoid vanishes, then the same is true for the reduced groupoid.\\

\noindent {\bf Notation:} For a groupoid $\gpd\toto M$ we denote $t,s:\gpd\to M$ as target and source maps, $m:\gpd\times_{M}\gpd \to \gpd$ the product of composable arrows on the groupoid, and $\epsilon:M\to \gpd$ the  units  of $\gpd$. Throughout the article we assume that the space of arrows $\gpd$ is connected. Multiplicative forms $\theta$ in $\gpd$ are defined as those differential forms for which the equation $m^*\theta=pr_1^*\theta+pr_2^*\theta$ holds, where  $pr_j:\gpd\times \gpd \to \gpd$ are the usual projection maps. In the case of an action of a Lie group $\Gp$ on a manifold $Q$, for any  $v\in \g=Lie(\Gp)$ we denote by $v_Q\in \mathfrak{X}(Q)$ the fundamental vector field corresponding to $v$. 
\\

\section{Preliminaries}\label{sec:Preliminares}
 	This section is dedicated to presenting the main objects of this article. Since we study the symmetries of cosymplectic groupoids and their infinitesimal counterpart, we present all the necessary definitions and tools to be used here.	
 	\subsection{Cosymplectic structures}
 	To establish the notation and as a quick reference, we describe the relation between a cosymplectic manifold and its Poisson structure; for a more detailed description, the reader can consult \cite{FI}.
 	
 	Let $(Q, \omega, \eta)$ be a \textbf{cosymplectic manifold}, i.e. $\omega\in \Omega_{cl}^2(Q)$ and $\eta\in \Omega_{cl}^1(Q)$, $\dim(Q)=2n+1$, and $\omega^n\wedge \eta$ defines a volume form. 
 	For this structure, the  Lichnerowicz morphism $\flat:\mathfrak{X}(Q)\to \Omega^1(Q)$  given by 
\begin{equation}\label{eq:flat}
\flat(X)=\imath_X\omega+(\imath_X\eta)\eta
\end{equation} 	
 	is an isomorphism. There is a remarkable vector field, called Reeb vector field, characterized by $\xi=\flat^{-1}(\eta)$, in particular it satisfies 
\begin{equation}\label{eq:Reeb}
\imath _{\xi}\omega=0\qquad \mbox{\ and\ }\qquad \imath_{\xi}\eta=1.
\end{equation}

\begin{remark}\label{rmk:flat} 
	The converse is also true and was originally stated in \cite[Prop 2]{Al}. The statement is as follows: if a manifold $Q$ is equipped with closed forms $\w\in \Omega^2(Q)$ and $\eta\in \Omega^1(M)$ for which $\flat$ in \eqref{eq:flat} is an isomorphism and there exists a vector field as in \eqref{eq:Reeb}, then $Q$ is odd-dimensional and $(Q,\w,\eta)$ is a cosymplectic manifold. $\diamond$
\end{remark}

\begin{example}\label{ex:submnfd}
	Let $(M, \omega)$ be a symplectic manifold of dimension $2m$. Let $i:N\hookrightarrow M$ be submanifold and $X$ be a symplectic vector field transverse  to $N$. Then, $(N, i^*\omega,  i^*(\imath_X\omega))$ is a cosymplectic manifold, in fact $d(i^*\w)=0$, $di^*(\imath_X\omega)=i^*(\mathcal{L}_X\omega-\imath_Xd\w)=0$, and because $X$ is traverse to $N$, $\dim(N)=2(m-1)+1$ and $(i^*\w)^{m-1}\wedge i^*(\imath_X\omega))$ is a volume form. The Lichnerowicz isomorphism is given by $\flat(Y)=i^*(\imath_{\tilde{Y}}\w+(\imath_{\tilde{Y}} \imath_X\w)\imath_X\w), $
	where $Y$ is $i$-related to $\tilde{Y}\in \mathfrak{X}(M)$.
\end{example}

\begin{example}\label{ex:product}
	Let $(M, \omega)$ be a symplectic manifold. Let  $L$ be a 1-dimensional manifold with coordinate chart $r$. Then $M\times L$ with $\text{pr}_1^*\omega $ and $\eta= \text{pr}_2^*dr$  is a cosymplectic manifold. The Reeb vector field is  $\xi:=\frac{\partial}{\partial r}$. 
	Note that there is  a regular symplectic foliation defined on $M\times L$, hence it is a Poisson manifold. 
	If  $f\in C^{\infty}(M\times L)$, then it is possible to define Hamiltonian vector fields on each leaf by $i_{X_{f_r}}\w=df_r$, where $f_r=f(-,r)$. Furthermore, on $M\times L$, 
	$$df_{(m,r)}=i_{X_{f_r}}\w+\dfrac{\partial f}{\partial r}\text{pr}_2^*dr=i_{X_{f_r}}\w+(\xi f)\eta$$ is satisfied. Moreover, for $f,g\in C^{\infty}(M\times L)$ such that $i_{X_g}\eta=0$, the Poisson bracket is $\w(X_f,X_g)$.
\end{example} 
The particular case studied in Example \ref{ex:product} motivates the definition of cosymplectic \textbf{Hamiltonian vector fields}. Thus, given a smooth function $f$, the cosymplectic Hamiltonian vector field is defined by $X_f=\flat^{-1}(df-(\xi f)\eta)$, or  equivalently by the unique vector field so that
\begin{equation}\label{eq:Hamvf}
\imath_{X_f}\omega=df-\xi(f)\eta, \qquad \mbox{\ and\ }\qquad \imath_{X_f}\eta=0. 
\end{equation}

Furthermore, the Poisson structure of Example \ref{ex:product} is a particular case of a more general result.	It is clear that $\ker(\eta)\subset TQ$ defines a regular integrable foliation $\F$ of corank 1. Moreover, the leaves of $\F$ are symplectic  with respect to the restriction of $\omega$. Indeed, there is a one to one correspondence between  cosymplectic structures on $Q$ and regular Poisson structures of corank 1 on $Q$ with a Poisson vector field transverse to the symplectic foliation (see \cite{GMP}).

 Hence, there is a Poisson structure $(Q, \pi_Q)$, whose bracket satisfies
	$\{f,g\}=X_fg=\w(X_f,X_g)$ because $i_{X_g}\eta=0$.
	As the Hamiltonian vector fields belong to $\Ker(\eta)$, we can verify that the bracket can also be written as
$\{f,g\}=\omega(\flat^{-1}(df),\flat^{-1}(dg)).$
	From the definition of $\xi$, we get $\mathcal{L}_\xi\w=0$ and $\mathcal{L}_\xi\eta=0$. By using the Cartan formula and \eqref{eq:Hamvf} we obtain
	$$\xi\{f,g\}=i_{X_f}i_{[\xi,X_g]}\w-i_{X_g}i_{[\xi,X_f]}\w=X_f(\xi g)-X_g(\xi f)=\{f,\xi g\}-\{g,\xi f\}.$$ 
	Therefore, the vector field $\xi$ is a Poisson vector field transverse to the leaves of $\F$.  
	
	A cosymplectic map $\phi:(Q,\omega, \eta)\to (Q', \omega', \eta')$ is a map between cosymplectic manifolds such that $\phi^*\omega'=\omega$ and $\phi^*\eta'=\eta$, in particular $\xi$ and $\xi'$ are $\phi$-related. If in addition $\phi$ is a diffeomorphism, then $i_{\phi_* X_f}\eta'=0$ and
	$$i_{\phi_* X_f}\w'=(\phi^{-1})^*(i_{X_f}\w)=d((\phi^{-1})^*f)-\xi'((\phi^{-1})^*f)\eta'.$$
	This relation says that $\phi_* X_f=X_{(\phi^{-1})^*f}$, and the map $\phi$ is a Poisson morphism between $\pi_Q$ and $\pi_{Q'}$.

 \subsubsection{Symmetries and reduction}
 \label{section:Symmetries and reduction}
 	For this section, we assume a Lie group $\Gp$ acting on a cosymplectic manifold on $(Q,\w,\eta)$ by cosymplectic diffeomorphisms. Thus, it induces a $\Gp$ Poisson action. In the case of a free and proper action, we obtain a reduced Poisson structure $\pi_r$ on $Q/\Gp$ with quotient map $q:Q\to Q/\Gp$ as Poisson morphism, that is $q_*\pi_Q=\pi_r$. By the correspondence between regular Poisson structures of corank 1 endowed with  a Poisson vector field transverse to the symplectic foliation and cosymplectic structures, we have that a Poisson $Q/\Gp$ reduction can be non-cosymplectic due to issues of rank or transversality. Nevertheless, there is a reduction procedure that preserves the cosymplectic structure, which is called \textbf{cosymplectic reduction} (see \cite{Al,CNY}). 
 	
 	Let $\Gp$ be a Lie group with Lie algebra $\g=Lie(\Gp)$. Let $(Q, \omega, \eta)$ be a cosymplectic manifold with an action  of $\Gp$  by cosymplectic morphisms. A smooth function $J:Q\to \g^*$ is called a \textbf{cosymplectic moment map} if it is $\Gp$-equivariant with respect to the coadjoint action, and  for any $v\in \g$ the Hamiltonian vector field $X_{\langle J, v\rangle}$ is the infinitesimal generator $v_Q\in \mathfrak{X}(Q)$. The value $0\in \g^*$ is called a \textbf{clean value} if $J^{-1}(0)\subset Q$ is a submanifold  and   $\ker(dJ)_x=T_x J^{-1}(0)$ for all $x\in J^{-1}(0)$ (in \cite{Al,MW}, it is called a \textbf{weakly regular value}). Clearly, a regular value is a clean value. Under these conditions there is a cosymplectic version of the classical Marsden-Weinstein  reduction theorem \cite{MW}:
 	\begin{theorem}[\cite{Al}]\label{thm:cosymp reduction}
 		Let $\Gp$ be a Lie group acting on a cosymplectic manifold $(Q, \omega, \eta)$ by cosymplectic automorphisms. Let $J:Q\to \g^*$ be a cosymplectic moment map such that $\xi(\langle J,v\rangle)=0$ for all $v\in \g$. If $0$ is a clean value of $J$, and  $\Gp$ acts freely and properly on $J^{-1}(0)$, then $Q_{red}=J^{-1}(0)/\Gp$ is a cosymplectic manifold. The cosymplectic structure  $(\omega_{red}, \eta_{red})$ is defined by 
 		$$i^*\omega=p^*\omega_{red}, \qquad i^*\eta=p^*\eta_{red},$$
 		where $i:J^{-1}(0)\to Q$ is the inclusion and $p:J^{-1}(0)\to Q_{red}$ is the projection.
 	\end{theorem}

\begin{remark}\label{rmk:0-regular}
Similarly to the symplectic case, when the action of $G$ on $Q$ is free, it follows that 0 is a regular value of $J$.
From the Lichnerowicz isomorphism we observe that $v_Q=\flat^{-1}(d\langle J,v\rangle)$. 
By the freeness of the action we conclude that $v_Q(x)\neq 0$ for any $v\neq 0\in \g$. Consequently $dJ_x:T_xQ\to \g^*$ is a surjective map. $\diamond$
\end{remark} 

Using a correspondence between cosymplectic moment maps on $Q$ and symplectic moment maps on  $Q\times \R$, we can study conditions under which cosymplectic moment maps exist.
\begin{proposition}
	\label{prop:existence of J}
	Let $\Gp$ be a Lie group acting on a connected cosymplectic manifold $(Q, \omega, \eta)$ by cosymplectic automorphisms such that $\imath_{v_Q}\eta=0, $ for all $ v\in \mathfrak{g}$. 
	If $H^1(Q)=0$ or $H^1_{Ch}(\mathfrak{g})=0$, then there exists a cosymplectic moment map $J:Q\to \mathfrak{g}^*$ such that $\xi(\langle J,v\rangle)=0$ for all $v\in \g$.
\end{proposition}
 
\begin{proof}
Let $i:Q\to Q\times \R$ be the inclusion at the 0 level, and $p_Q$ and $p_\R$ be the canonical projections from $Q\times \R$ to $Q$ and $\R$, respectively. Consider  $Q\times \mathbb{R}$ with the symplectic form $\tilde{\w}=p_Q^*\w+p_\R^*dr\wedge p_Q^*\eta$, where $r$ is the standard coordinate on $\R$. There is a trivial extension of the $G$-action on $Q$ to a $G$-action on $Q\times \R$ given by $g\cdot(q,r)=(g\cdot q,r)$.  We claim that there exists a one-to-one correspondence between cosymplectic moment maps on $(Q,\w,\eta)$ satisfying $\xi\langle J,v \rangle=0$, and symplectic moment maps on $(Q\times \R,\tilde{\w})$.

Note that the vector fields $v_Q$ and $v_{Q\times \R}$ are $p_Q$-related. Thus, if $J:Q\to \g^*$ is a cosymplectic moment map such that $\xi\langle J,v \rangle=0$, then $\tilde J:Q\times \R\to \g^*$ given by $\tilde J(q,r)=J(q)$ satisfies 
$$\imath_{v_{Q\times \R}}\tilde \w=\imath_{v_{Q\times \R}}(p_Q^*\omega+p_\R^*dt\wedge p_Q^*\eta)=p_Q^*(\imath _{v_Q}\omega)=p_Q^*d\langle J,v\rangle=d\langle \tilde J, v\rangle.$$
On the other hand, given a symplectic moment map $\tilde J:Q\times \R\to \g^*$, we verify
$\partial_r \langle \tilde{J},v \rangle =\imath_{\partial_r}\imath_{v_{Q\times \R}}\tilde{\w}=p_Q^*(\imath_{v_Q}\eta)=0.$ Thus, $\tilde{J}$ does not depend on $r$, and we can define $J:Q\to \g^*$ such that $\tilde J=p_Q^*J$. If $\tilde \xi$ is the unique $p_\R$-vertical vector field which is $p_Q$-related with $\xi$, then  $\imath_{\tilde{\xi}}\imath_{v_{Q\times \R}}\tilde{\w}=0.$ Therefore, $\xi\langle J,v\rangle =\tilde \xi \langle \tilde J,v\rangle= 0.$
Consequently, from this correspondence we conclude that the obstruction to the existence of $\tilde J$ is the same as $J$. The result follows from \cite[Prop. 4.5.17]{ORbook} and from the homotopy equivalence between $Q$ and $Q\times \R$.
\end{proof}	

The proof of the previous result relies on the construction of a symplectic structure on $Q \times \mathbb{R}$, known as the {\it symplectization procedure}. As briefly mentioned, one can derive a symplectic moment map on $Q \times \mathbb{R}$ from a cosymplectic moment map on $Q$. In fact, the symplectic reduction $(Q//\Gp\times \mathbb{R},\tilde{\w}_{red})$
at the zero level induces the cosymplectic reduction $(Q//\Gp,\w_{red},\eta_{red})$ via 
$$\w_{red}=i^*\tilde{\w}_{red},\qquad \eta_{red}=i^*(\imath_{\partial r}\tilde{\w}_{red}),$$
where $i:Q//\Gp\to Q//\Gp\times \mathbb{R}$ is the natural inclusion $i(m)=(m,0).$

\subsection{Oversymplectic groupoids}\label{sec:intro-cs-gpd}
Let $\gpd\toto M$ be a Lie groupoid endowed with a closed multiplicative 2-form $\omega\in \Omega^2(\gpd)$, such that $\rank(\omega_x)=2\dim(M)$ for all $x\in M$. The pair $(\gpd, \omega)$ is called an \textbf{oversymplectic groupoid}, and was originally studied in \cite{BCWZ} in the context of integration of Dirac structures. 
The multiplicativity of the 2-form descends to infinitesimal data. To this end, we first recall some basic notions from the theory of \textbf{infinitesimal multiplicative forms},  following \cite{BC}. Let $(A,\rho,[\cdot,\cdot])$ be a Lie algebroid, and let $\mu:A\to \Lambda^{k-1}T^*M$  be  a vector bundle map. The map  $\mu$ is called an IM $k$-form \footnote{The definition used here is actually known as {\it closed IM forms}, and it is not the most general one. However, since we will work only with closed forms, it is enough to consider this simpler case instead of the general case.} if

\begin{align}
 		\label{eq:IM-skew}
 	\imath_{\rho(a)}\mu(b)&=-\imath_{\rho(b)}\mu(a),\\ \label{eq:IM-bckt}
 	\mu([a,b])&=\mathcal{L}_{\rho(a)}\mu(b)-\imath_{\rho(b)}d\mu(a),	
 	\end{align}
 	for any $a,b\in \Gamma(A)$.

A closed multiplicative form $\alpha\in \Omega_{cl}^k(\mathcal{G})$ on a Lie groupoid $\gpd$ with Lie algebroid $A$ naturally induces an IM $k$-form $\mu:A\to \Lambda^{k-1}T^*M$ by  
 	$$\mu:A\to \Lambda^{k-1}T^*M \qquad \langle \mu(u), X_1,\ldots, X_{k-1}\rangle:=\alpha(u,d\epsilon X_1,\ldots, d\epsilon X_{k-1}),$$
where $u\in A$ and $X_i\in TM$. Equivalently, this definition can be written  as $\epsilon^*(\imath_{u^R}\alpha)=\mu(u)$, where $u^R$ is the right invariant vector field on $\gpd$ induced by $u\in \Gamma(A)$. In particular, any oversymplectic groupoid $(\mathcal{G}, \omega)$  induces an IM 2-form $\mu:A\to T^*M$ on its Lie algebroid, satisfying the conditions \eqref{eq:IM-skew}-\eqref{eq:IM-bckt}. 
 
 \subsubsection{Symmetries and reduction of oversymplectic groupoids}\label{sec:over-gpd}
 In \cite{BC2}, the authors develop the theory of symmetries and reduction of closed IM 2-forms. Here we summarize this notion in the context of oversymplectic groupoids, with the aim of applying it to cosymplectic groupoids. Let $A$ be a Lie algebroid equipped with an IM 2-form $\mu:A\to T^*M$, and let $\Gp$ be a Lie group acting on $A$ by vector bundle morphisms covering the action of $\Gp$ on $M$, such that $\mu$ is $\Gp$-equivariant. Recall that the $\Gp$-action on $M$ lifts to a Hamiltonian action on $(T^*M,\omega_{{can}})$, with canonical moment map $J_{can}:T^*M\to \mathfrak{g}^*$ given by
\[
J_{{can}}(\alpha)(u) = \iota_{u_M}\alpha, \qquad u \in \mathfrak{g}={Lie}(\Gp),
\]
where $u_M$ is the infinitesimal generator of $u$.

 
\begin{proposition}[Lemma 3.1 and Proposition 3.2 in \cite{BC2}]\label{prop:over-mm}
The map $J = J_{{can}} \circ \mu : A \to \mathfrak{g}^*$ is a Lie algebra morphism, where $\mathfrak{g}^*$ carries the abelian Lie algebra structure. If $(A,\mu)=Lie(\gpd,\omega)$ for an oversymplectic groupoid $(\gpd,\omega)$ that is source-simply-connected, the action integrates to a $\Gp$-action on $\gpd$, and $J$ integrates to a $\Gp$-equivariant groupoid morphism $\mathcal{J}:\gpd \to \mathfrak{g}^*,$
satisfying
\[
\iota_{u_\gpd}\omega = -d\mathcal{J}(u), \qquad u\in \mathfrak{g}.
\]
\end{proposition}
Since $J$ is a Lie algebroid morphism, its zero level set $J^{-1}(0)$ is a $\Gp$ invariant vector subspace (pointwise). 
If $J^{-1}(0)$ is a vector subbundle of $A$ and the action of $\Gp$ on $M$ is free and proper, one can consider the reduced vector bundle $A_{{red}} \to M/\Gp$, 
uniquely characterized by the property that its pullback $p^*(A_{{red}})$ is naturally isomorphic to $J^{-1}(0)$, 
and such that the pullback of sections identifies sections of $A_{{red}}$ with the $\Gp$-invariant sections of $J^{-1}(0)$. Thus, the vector bundle $A_{{red}}$ inherits a unique Lie algebroid structure 
$(A_{{red}},[\cdot,\cdot]_{{red}},\rho_{{red}})$, so that:
\begin{itemize}
    \item The pullback of sections $p^*:\Gamma(A_{red})\to \Gamma(J_A^{-1}(0))$ preserves the Lie brackets.
    \item For any $a \in \Gamma(A_{red})$, the vector fields $\rho_{red}(a)$ and $\rho(p^*a)$ are $p$-related via the quotient map $p:M \to M/\Gp$. 
\end{itemize}
As a consequence, the following result holds:
\begin{proposition}[Lemma 4.1 and Proposition 4.3 in \cite{BC2}]\label{prop:over-0-level}
If $J^{-1}(0)$ is a subbundle of $A$, then it is a Lie subalgebroid. Moreover, if the action of $\Gp$ on $M$ is free and proper, one obtains a reduced bundle 
$A_{red} \to M/\Gp$, which is a Lie algebroid equipped with an IM 2-form $\mu_{red}: A_{red} \to T^*(M/\Gp)$.
\end{proposition}
The reduced IM-form is explicitly constructed from the restriction $\mu_0: J^{-1}(0)\to J_{can}^{-1}(0)$ by the inclusion map $J^{-1}(0)\to A$. Note that $\mu_0$ is $\Gp$-equivariant, and hence it is possible to define the bundle map $\mu_{red}:A_{red}\to T^*(M/\Gp)$ by the identity
    		\begin{equation}\label{eq:mu-red}
    		p^*(\mu_{red}(a))=\mu_0(p^*(a)),\qquad a\in \Gamma(A_{red}).
    		\end{equation}
    		The following diagram illustrates the pullback maps that appear in this equation.
    	$$\xymatrix@R=10pt@C=10pt{
    		J^{-1}(0)\simeq p^*A_{red} \ar[rr]\ar[dd]_{} \ar[rd]_{\mu_0} & &  A_{red}\ar'[d]^{}[dd] \ar[rd]^{\mu_{red}}  \\
    		& J^{-1}_{can}(0) \ar[rr]_{}   & & T^*(M/G)  \\
    		M \ar[rr]_{p}  & & M/G,   \\ 
    	} 
    	$$	

In the case where $(A,\mu) = Lie(\gpd,\omega)$ for an oversymplectic groupoid $(\gpd,\omega)$ that is source-simply-connected, the $\Gp$-action on $A$ integrates to a $\Gp$-action on $\gpd$, and $J$ integrates to a $\Gp$-invariant map $\mathcal{J}:\gpd\to\mathfrak{g}^*$, so that the following result holds\footnote{The integration of the $\Gp$-action on $A$ and of the map $J:A\to\mathfrak{g}^*$ to a $\Gp$-action on $\gpd$ and a groupoid morphism $\mathcal{J}:\gpd\to \mathfrak{g}^*$ follows from Lie’s second theorem.
}:
\begin{proposition}[Lemma 5.1 and Proposition 5.2 in \cite{BC2}]\label{prop:over-integration}
If $0$ is a clean value 
 of $\mathcal{J}$, and the $G$-action on $M$ (lifting to $\gpd$) is free and proper, then $\mathcal{J}^{-1}(0)$ is a wide subgroupoid of $\gpd\toto M$ with 
\[
Lie(\mathcal{J}^{-1}(0)) = J^{-1}(0).
\]
Moreover, the reduction of $\mathcal{J}^{-1}(0)$ by the $G$-action \footnote{It is worth to remark that if $\Gp$  acts freely and properly on $M$, then $\Gp$ also acts freely and properly on $\gpd$ (see \cite[Prop. 4.4]{FOR}).}	
 is a presymplectic Lie groupoid $(\gpd_{red},\omega_{red}) \toto M/G$ such that
\[
Lie(\gpd_{red}) = A_{red}, \qquad Lie(\omega_{red}) = \mu_{red},
\]
where $\omega_{red}$ is uniquely characterized by the identity $i_0^*\omega = p_0^*\omega_{red}$, with $i_0:\mathcal{J}^{-1}(0)\rightarrow \gpd$ the inclusion and $p_0:\mathcal{J}^{-1}(0)\to \gpd_{red}$ the quotient map.
\end{proposition}
To avoid confusion, we denote the action of $\Gp$ on $M$, $A$, and $\gpd$ by 
$\varphi^M$, $\varphi^A$, and $\varphi^\gpd$, respectively.

 \subsubsection{Cosymplectic groupoids}\label{sec:cs-gpd}  
This section is devoted to the introduction of the main object of the present article, namely cosymplectic groupoids. To this end, we make use of the notion of oversymplectic groupoids studied earlier.

\begin{definition}
    \label{def:CoymplecticGroupouid}
    A cosymplectic groupoid is an oversymplectic groupoid equipped with a closed, multiplicative $1$-form $\eta \in \Omega^1(\gpd)$ such that $(\gpd,\omega,\eta)$ is a cosymplectic manifold. 
\end{definition}
 Recall that cosymplectic structures can be described using the $\flat$-morphism as in \eqref{eq:flat}; however, this morphism is not enough to ensure the multiplicativity condition for $\eta$.  Thus, using $\flat$ to construct a cosymplectic groupoid from an oversymplectic groupoid proves to be of limited usefulness. However, there is a way to study when it is possible to construct cosymplectic groupoids from oversymplectic groupoids, which we describe below.

 Let $(\gpd, \omega)$ be a proper, orientable oversymplectic groupoid, and suppose that the foliation determined by $\ker\omega$ is simple, namely, the leaf space $\gpd/\ker \omega$ is a smooth manifold and the map $\gpd\to\gpd/\ker \omega$ is a surjective submersion. Furthermore, as the corank of $\omega$ is 1, i.e., the symplectic foliation is  regular of codimension 1; then  there exists an $S^1$-central extension  	 	 $$1\to M\times S^1\to (\gpd,\w)\xto{q} (\Sigma_1,\w_1)\to 1,$$ 
 	 	 where $(\Sigma_1, \w_1)$ is a symplectic groupoid with $q^*\w_1=\w$.
         
Hence, $\gpd$ becomes an $S^1$-principal bundle, where $\partial_\theta$ denotes the infinitesimal generator of the $S^1$-action on $\gpd$. Moreover, there exists a multiplicative form $\eta \in \Omega^1(\gpd)$ satisfying $\imath_{\partial_\theta}\eta=1$. The form $\eta$ is related to a \textit{multiplicative Ehresmann connection} (via the kernel of $\eta$), and the curvature of this connection is $\Omega=d\eta$. We refer the reader to \cite{FI,FM} for further details on this construction. The 2-form $\Omega$ is multiplicative and $q$-basic i.e., there exists a form $\Omega_1\in \Omega_{cl}(\Sigma_1)$ such that $q^*\Omega_1=\Omega$.  The {\bf multiplicative Chern class} is defined as the cohomology class  associated with $\Omega_1$ (in the multiplicative de Rham cohomology). 

One of the most interesting results in \cite{FI} is that this Chern class is an obstruction to an oversymplectic groupoid becoming cosymplectic:

\begin{theorem}\label{thm:CC-cosym}\cite[Theorem 4.12]{FI}
Let $(\gpd,\omega)$ be a corank 1, orientable, proper oversymplectic groupoid over $M$, whose  foliation $\ker \omega$ is simple.  Then $(\gpd,\omega)$ admits a multiplicative 1-form $\eta$ making it cosymplectic if and only if the multiplicative Chern class vanishes.
\end{theorem} 	 	  

Some consequences of the definition of a cosymplectic groupoid are the following:

\begin{proposition}[Proposition 2.7 and Proposition 2.9 in \cite{FI}]
Let $(\gpd,\omega,\eta)$ be a cosymplectic groupoid over $M$, and let $\Sigma$ be the leaf of the symplectic foliation containing the units $M \subset \gpd$. Then
\begin{enumerate}
\item  $\dim \gpd=2\dim M+1$;
\item $\Ker \eta\subset T\gpd$ is a multiplicative distribution;
\item the induced Poisson structure on $\gpd$  is multiplicative;
\item the Reeb vector field $\xi\in \mathfrak{X}(\gpd)$ is a bi-invariant Poisson vector field;
\item $(\Sigma,\omega|_{\Sigma})\toto M$ is a symplectic groupoid, where  $\omega|_{\Sigma}$ is the restriction to $\Sigma$ of $\w$;
\item $(\Sigma,\omega|_{\Sigma})$ integrates the Poisson structure $(M, \pi_M)$, where $\pi_M=t_* \pi_\mathcal{G}$.
\end{enumerate}
\end{proposition}

Following the  two canonical constructions of cosymplectic manifolds in Examples \ref{ex:submnfd} and \ref{ex:product}, it is possible to produce examples of cosymplectic groupoids.
 	\begin{example}
 		\label{ex:restriction gpd}
 		Let $(\gpd,\w)\toto M$ be a symplectic groupoid. Consider $i:\gpd_0\hookrightarrow \gpd$ a Lie subgroupoid over $M_0\hookrightarrow M$, both submanifolds of codimension 1. If there exists a multiplicative symplectic vector field $X$ transverse to $\gpd_0$, then $(\gpd_0,i^*\w,i^*(\imath_X\w))\toto M$ is a cosymplectic groupoid. For a concrete example, consider a manifold $Q$, $\gpd=T^*(Q\times \R)$ and $\w=\w_{can}$. Let $\gpd_0$ be the product of $T^*Q\toto Q$ and $\R\toto\{*\}$. If $r$ is the coordinate of the base of $T^*\R\to \R$, then the vector field $X=\partial_r$ is multiplicative, symplectic and transverse to $\gpd_0$.
\end{example}
	
\begin{example}
 		\label{example: TMxS}
 		Associated with the construction of  Example~\ref{ex:product} we can consider two  product groupoids.  Let $(\gpd_0,\w)\toto M$ be a symplectic groupoid. Consider 		$\gpd_1=\gpd_0\times S^1 \toto M$  and $\gpd_2=\gpd_0\times S^1\toto M\times S^1$, where the structural maps of $\gpd_1$ and $\gpd_2$ come from the groupoids $S^1\toto \{*\}$ and $S^1\toto S^1$, respectively. The cosymplectic forms $pr_{\gpd_0}^*\w$ and $pr_{S^1}d\theta$ are multiplicative in $\gpd_1$ making it a cosymplectic groupoid. However,  $(\gpd_2, \text{pr}_{\gpd_0}^*\omega, \text{pr}_{S^1}^*d\theta)$ is not a cosymplectic groupoids by dimensional issues.
 	\end{example}

\subsubsection{The infinitesimal counterpart}
We now describe the infinitesimal counterpart of a cosymplectic groupoid $(\gpd,\omega,\eta)$. As recalled above, the 2-form $\omega$ induces an IM 2-form $\mu:A \to T^*M$. In addition, using the identities 
\eqref{eq:IM-skew}--\eqref{eq:IM-bckt}, one obtains an IM 1-form $\nu:A \to \mathbb{R}$, associated with $\eta$, which satisfies
\begin{equation}\label{eq:nu}
 	\nu([a,b])=\rho(a)\nu(b)-\rho(b)\nu(a).
 	\end{equation}

Since the Reeb vector field $\xi$ is bi-invariant on $\gpd$, it can be expressed as 
$\xi = \overrightarrow{\zeta} = \overleftarrow{\zeta}$ for a section 
$\zeta \in \Gamma(A)$ lying in $\Ker \rho$ and in the center of the Lie algebra $(A,[\cdot,\cdot])$.

\begin{proposition}[Proposition 3.1 in \cite{FI}]\label{Prop:CentralExtension}
The Lie algebroid $A\to M$ of a cosymplectic groupoid $(\gpd,\w,\eta)$ is a central extension 
\begin{equation}\label{eq:c-ext}
0\to \langle \zeta \rangle \to A\overset{\mu}{\to} T^*M\to 0
\end{equation}
where $T^*M$ is equipped with the cotangent Lie algebroid structure associated with the Poisson manifold $(M,\pi_M)$, and $\langle \zeta \rangle$ is the trivial line bundle generated by $\zeta \in \Gamma(A)$. This extension has a natural splitting given by $\underline{\nu}:A\to \langle \zeta \rangle$ with $\underline{\nu}(a)=\nu(a)\zeta$.
\end{proposition} 	

Motivated by the result in Proposition \ref{Prop:CentralExtension}, we adopt the notion of an \textbf{infinitesimal multiplicative cosymplectic algebroid} (IMcs algebroid for short).

\begin{definition}\label{def:IMCS}
An IMcs algebroid over a Poisson manifold $(M,\pi_M)$ is a Lie algebroid $A$ 
equipped with an IM 2-form $\mu$, an IM 1-form $\nu$, and a section $\zeta \in \Gamma(A)$, 
so that \eqref{eq:c-ext} defines a central extension with splitting $\underline{\nu}(a) = \nu(a)\zeta$. 
We denote such an IMcs algebroid by $(A,\mu,\nu,\zeta)$.
\end{definition} 
Directly, from the definition of central extensions \eqref{eq:c-ext}, $\mu$ is surjective with kernel $\langle \zeta \rangle$. Moreover, the map $\mu\oplus \nu$ is an isomorphism of $A$ and $T^*M\oplus \mathbb{R}$. Under this identification, the anchor map and bracket are given by
 	\begin{equation}
 	\label{prop:centralext A}
 	\rho(df,\lambda)=X_f, \qquad [(df_1, g_1), (df_2, g_2)]=(d \{f_1,f_2\}_M, X_{f_1}(g_2)-X_{f_2}(g_1)),
 	\end{equation}
 	where $X_{f_i}$ are the Hamiltonian vector fields defined from the Poisson structure $(M, \pi_M)$. 

We conclude this section with a brief remark on cosymplectic reduction via the symplectization procedure.

\begin{remark}
The cosymplectic reduction described at the end of Section \ref{section:Symmetries and reduction} through the \textit{symplectization procedure} also applies to cosymplectic groupoids and their infinitesimal analogues.  From a cosymplectic groupoid $(\gpd \toto M, \omega, \eta)$ one can construct the symplectic groupoid 
$
(\gpd \times \mathbb{R} \toto M \times \mathbb{R}, \ \tilde{\omega} = \omega + dr \wedge \eta),
$
together with its IM-form $\tilde{\mu} : \tilde{A} \to T^*(M \times \mathbb{R})$, where $\tilde{A} = Lie(\gpd \times \mathbb{R})$, satisfying 
\[
\tilde{\mu}(u)(X) = \omega(u,d\epsilon X) + \eta(u).
\]

However, in this construction the central extension associated with $\gpd \toto M$ (and its Lie algebroid) is not explicitly encoded in $\gpd \times \mathbb{R} \toto M \times \mathbb{R}$, but depends solely on the pair of multiplicative forms $(\omega,\eta)$ (or equivalently, on their IM-forms). Moreover, the symplectization procedure may fail to capture certain geometric aspects, as discussed in Sections~\ref{sect:consecuencias2} and \ref{sect:consecuencias3}. For these reasons, we restrict ourselves to the cosymplectic reduction considered in the present manuscript.
\end{remark}

 	\section{Symmetries and reduction of cosymplectic groupoids}
 	\label{sect: MainTheorem}
 	
    We present here our first main contribution, which is a reduction theorem for the cosymplectic groupoid using a Marsden-Weinstein approach. We begin by pointing out why a Poisson reduction is not convenient when working with cosymplectic groupoids. This motivates the introduction of a moment map in cosymplectic groupoids and thus using Theorem \ref{thm:cosymp reduction} to carry out the reduction.
 	
 	Let $(\gpd,\omega,\eta)\toto M$ be a cosymplectic groupoid and let $(\gpd,\pi_\gpd)\toto M$ be its associated regular Poisson groupoid of corank 1. Let $\Gp$ be a Lie group acting on $\gpd$ and $M$, such that the actions are compatible with the structural maps and the action is by cosymplectic diffeomorphisms. Note that the action is also by Poisson morphisms. Then we obtain Poisson diffeomorphisms on $\gpd$ and on $M$. If the actions $\Gp\action M$ and $\Gp\action \gpd$ are free and proper, then $\gpd/\Gp\toto M/\Gp$ is a Poisson Lie groupoid.  Furthermore,  the Poisson structures can be reduced so that the respective quotient maps and $\gpd/\Gp\xto{\bar{t}} M/\Gp$ are Poisson maps. In particular the quotient maps and the target maps commutes
 
 	$$\xymatrix{
 		(\gpd, \pi_\gpd) \ar[d]_{t}  \ar[r]& (\gpd/\Gp, \pi_r) \ar[d]^{t_r}\\
 		(M, \pi_M)\ar[r] & (M/\Gp, \pi_r).
 	}$$
 	It is important to highlight that this reduction may not be a cosymplectic groupoid, as the rank condition may not hold (cf. \cite[Prop. 18]{GMP}). In  the following example, we start with a cosymplectic groupoid and after performing a Poisson reduction we obtain a groupoid that is not a regular Poisson groupoid of corank 1, and consequently it is not a cosymplectic groupoid.
 	\begin{example} 
 		\label{Example: ReduNonCos}  Consider the cosymplectic groupoid $\gpd=T^*\R^n\times S^1 \toto \R^n$ of Example \ref{example: TMxS}, with $\w=\w_{can}$. A point in $T^*\R^n$ is presented as a pair $(q, p)$ consisting of two points $q,p\in \R^n$ using that $T^*\R^n\cong \R^n\times \R^n$ canonically. Let $\R^{n-k}\action \R^n$  be the translation in the last $n-k$ coordinates, and  $\varphi$  the $\R^{n-k}$-action on $T^*\R^n\times S^1$ defined by $\varphi_{\alpha}( (q, p),\beta)=((\alpha\cdot q,p) , \beta)$. Since $\R^{n-k}$ acts on $\gpd$ by cosymplectic groupoid isomorphisms, we obtain the Poisson groupoid $\gpd/\R^{n-k}\toto \R^n/\R^{n-k}$. Note that $\dim(\gpd/\R^{n-k})=n+k+1$ and $\dim(\R^n/\R^{n-k})=k$, thus for $k\neq n$ the reduction is not a cosymplectic groupoid. 
 	\end{example}
 	
 		Thus, given a cosymplectic groupoid $(\gpd,\omega,\eta)\toto M$ and a Lie group $\Gp$ acting on $\gpd$ by groupoid automorphisms and cosymplectic diffeomorphisms, if we want to study cosymplectic reduction, we must consider reduction by a  moment map. Let $\Jmm:\gpd\to \g^*$ be a cosymplectic moment map. In addition, we assume that $\Jmm$ is a morphism of Lie groupoids, where $\g^*$ is considered as an Abelian group with respect to addition.  If we suppose that 0 is a clean value of $\Jmm$ we have the following result.
        
 		\begin{lemma}
 			\label{lem:Jlevel0subgroupoid}
 			Let $\Jmm:\gpd\to \g^*$ be a morphism of groupoids, where $\g^*$ is considered as an Abelian group with respect	to addition. If $0$ is a clean value of $\Jmm$, then $\Jmm^{-1}(0)\subset \gpd$ is a wide Lie subgroupoid.
 		\end{lemma}
 		\begin{proof}
		For a pair of composable arrows $(g,h)$, we have $\Jmm (gh)=\Jmm(g)+\Jmm(h)$. Hence, $\Jmm^{-1}(0)$ is closed under the multiplication and inverse on $\gpd$. Also, for any $x\in M$, $\epsilon(x)\in \Jmm^{-1}(0)$, so the subgroupoid  is wide. What remains to be demonstrated is that the restrictions of the maps $s,t$ to $\Jmm^{-1}(0)$ are submersions. If $i:\Jmm^{-1}(0)\to \gpd$ is the inclusion map, then 
 $$\ker(d(i^*s))_g=\ker(ds)_g\cap T_g\Jmm^{-1}(0)=\ker(ds)_g\cap \ker(d\Jmm)_g,\quad \text{  for all }g\in \gpd.$$
 Thus, $\rank(d(i^*s))=\dim(\Jmm^{-1}(0))-(\dim(A)+\dim(\Jmm^{-1}(0))-\dim(\gpd))=\dim(M)$.
 		\end{proof}
 		A similar version of this result is proved in Proposition \ref{prop:over-integration} when $\Jmm$ is the integration of a Lie algebroid morphism. On the other hand, as a direct consequence of \cite[Prop. 4.4]{FOR}, it is enough to assume that the action is free on the units of the groupoid to guarantee that $0$ is a regular value of $\Jmm$, a fact stated in the following lemma.
	\begin{lemma}
	\label{lemma: FreeactionRegularvalue}
Let $\Gp$ be a Lie group acting on a cosymplectic groupoid $(\gpd, \omega, \eta)\toto M$ by cosymplectic groupoid automorphisms.  Let $\Jmm:\gpd\to \g^*$ be a cosymplectic moment map, and $\xi(\langle \Jmm,v\rangle)=0$ for all $v\in \g$. If  $\Gp$ acts freely on $M$, then $0$ is a regular value of $\Jmm$.
	\end{lemma}
	\begin{proof}
Since the action is free on $M$, then it is also free on $\gpd$. Thus we conclude the result using Remark \ref{rmk:0-regular}.
	\end{proof}

 		 	Motivated by Theorem \ref{thm:cosymp reduction}, we propose to address the cosymplectic structure of $\Jmm^{-1}(0)/\Gp$, and see if it has a compatible groupoid structure. Thus,  we present the first main theorem of these notes, which is a Marsden–Weinstein reduction of cosymplectic groupoids.

 		\begin{theorem}
 			\label{thm: MainTheorem}
 			Let $\Gp$ be a Lie group acting on a cosymplectic groupoid $(\gpd, \omega, \eta)\toto M$ by cosymplectic groupoid automorphisms. Let $\Jmm:\gpd\to \g^*$ be a cosymplectic moment map such that it is a Lie groupoid morphism, and $\xi(\langle \Jmm,v\rangle)=0$ for all $v\in \g$. If $\Gp$ acts  freely and properly on $M$, then $\gpd_{red}=\Jmm^{-1}(0)/\Gp$ is a cosymplectic groupoid over $M/G$. The cosymplectic structure  $(\omega_{red}, \eta_{red})$ is defined by 
 			$$i^*\omega=P^*\omega_{red}, \qquad i^*\eta=P^*\eta_{red},$$
 			where $i:\Jmm^{-1}(0)\to \gpd$ is the inclusion and $P:\Jmm^{-1}(0)\to \gpd_{red}$ is the projection.
 		\end{theorem}

\begin{proof}
	From Lemma \ref{lemma: FreeactionRegularvalue},  it follows that $\Jmm^{-1}(0)$ is a smooth manifold. 
	The equivariance of $\mathcal{J}$, together with the structural maps of the groupoid and the freeness and properness of the $\Gp$-action, 
	give rise to the reduced Lie groupoid $\gpd_{red} = \Jmm^{-1}(0)/\Gp\toto M/\Gp$. 
	Moreover, the quotient maps 
	$
	(P,p): (\mathcal{J}^{-1}(0) \toto M) \longrightarrow (\gpd_{red} \toto M/\Gp)
$
	define a morphism of Lie groupoids. 
	
	Moreover, by Theorem \ref{thm:cosymp reduction}, the manifold $(\gpd_{red}, \w_{red},\eta_{red})$ characterized by $i^*\omega=P^*\omega_{red}$ and  $i^*\eta=P^*\eta_{red},$ is a cosymplectic manifold. 
	To show that $\gpd_{red}\toto M/\Gp$ is a cosymplectic groupoid it remains to verify that $\omega_{red}$ and $\eta_{red}$  are multiplicative. Since the quotient map commutes with the structural groupoids maps, and by Lemma \ref{lem:Jlevel0subgroupoid} we have  the next commutative diagram 
	$$\xymatrix@R=10pt@C=10pt{
		\Jmm^{-1}(0)^{(2)} \ar[rr]_{m_0}\ar[dd]_{(i,i)} \ar[rd]_{(P,P)} & & \Jmm^{-1}(0) \ar'[d]^{i}[dd] \ar[rd]^{P}  \\
		& \gpd_{red}^{(2)} \ar[rr]_{m_{red}\hspace{5mm}}   & & \gpd_{red}  \\
		\gpd^{(2)} \ar[rr]_{m}  & & \gpd.   \\ 
	}
	$$	
	
	From which we conclude,
	$$P^*m_{red}^*\eta_{red}=m_0^*i^*\eta=i^*(pr_1^*\eta+pr_2^*\eta)=P^*(pr_1^*\eta_{red}+pr_2^*\eta_{red}),$$
	where $m_0$ is the restriction of the product in $\gpd$ to $\Jmm^{-1}(0)$. Because $P$ is a surjective submersion, then $m_{red}^*\eta_{red}=pr_1^*\eta_{red}+pr_2^*\eta_{red}$. Analogously, it is proved that $\omega_{red}$ is multiplicative.
\end{proof}

\begin{example}
	In Example \ref{ex:restriction gpd}, consider a Lie group $\Gp$ acting on $(\gpd,\omega)$ by  symplectic groupoid automorphisms. Let $\Jmm:\gpd \to \mathfrak{g}^*$ be a symplectic moment map such that it is a Lie groupoid morphism. Suppose that $\gpd_0$ and $M_0$ are  closed under the action of $\Gp$. In addition, we assume that the vector field $X\in \mathfrak{X}(\gpd)$ is $G$-invariant, and $X\langle \Jmm,v\rangle=0$ for all $v\in \g$. Therefore, the map $\Jmm_0:\gpd_0 \xto{i_0}\gpd \xto{\Jmm} \mathfrak{g}^*$ is a cosymplectic moment map and a Lie groupoid morphism. If $\Gp$ acts freely and properly on $M$, and $\Gp$ acts properly on $M_0$, then $(\gpd_0)_{red}=\Jmm_0^{-1}(0)/\Gp\toto M_0/\Gp$ is a cosymplectic groupoid. Giving rise to the commutative diagram 	$$
	\xymatrix@R=10pt@C=10pt{
		\Jmm_0^{-1}(0) \ar[rr]^{}\ar[dd]_{} \ar[rd]^{} & & \Jmm^{-1}(0) \ar'[d]^{}[dd] \ar[rd]^{}  \\
		& \gpd_0 \ar[rr]^{i_0\hspace{6mm}}   & & \gpd  \\
		(\gpd_{0})_{red} \ar[rr]_{i}  & & \gpd_{red},  \\ 
	}
	$$
	where $(\gpd_{red}=\Jmm^{-1}(0)/G, \omega_{red})\toto M/G$ is the reduced symplectic groupoid \cite{FOR}. Moreover, there exists a  $\Gp$-invariant vector field $X_0\in \mathfrak{X}(\Jmm^{-1}(0))$, which is  $i$-related to $X$. Also, $X_0$ is projectable  to a vector field $\bar{X}_0\in \mathfrak{X}(\gpd_{red})$. Thus, the reduced cosymplectic groupoid  coincides with $((\gpd_0)_{red},i^*\w_{red},i^*(\imath_{\bar{X_0}}\w_{red}))$.
\end{example}

\begin{example}
In Example \ref{Example: ReduNonCos}, consider the linear momentum associated with $\R^{n-k}\action T^*\R^n$. Thus,  the cosymplectic moment map $\Jmm:\gpd\to \R^{n-k}$ is defined by $\Jmm((q,p),\beta)=(p_{k+1},\ldots, p_{n})$. Then, $\Jmm^{-1}(0)\cong T^*\R^k\times\R^{n-k}\times  S^1$ canonically, and $\gpd_{red}=T^*\R^k\times S^1$. Note that the inclusion $i:\Jmm^{-1}(0)\hookrightarrow T^*\R^n\times S^1$ is by considering the zero section in the last $n-k$ coordinates, that is,  $$((q_1,p_1,\ldots,q_k,p_k),(q_{k+1},\ldots,q_n),\beta)\to((q_1,p_1,\ldots,q_k,p_k,q_{k+1},0,\ldots,q_n,0),\beta).$$ So $i^*(pr^*_{T^*\R^n}\omega_{can})=\sum_{j=1}^k dq_j\wedge dp_j$. Therefore $\omega_{red}=pr_{T^*\R^k}^*\sum_{j=1}^kdq_j\wedge dp_j$ and $\eta_{red}=pr_{S^1}d\theta$.
\end{example}

    \section{Consequences of the main theorem}
    \label{sect: consecuencias}
    This section is devoted to applying Theorem \ref{thm: MainTheorem}  to specific situations concerning the multiplicativity of cosymplectic forms in a Lie groupoid.   In particular, we are inspired by the methods presented in \cite{BC2,FOR} but in the context of cosymplectic structures. We also explore some of the main results in \cite{FI} when symmetries and reduction are considered.

    		\subsection{Symmetries and reduction at the infinitesimal level}
    		\label{sect:consecuencias1}
Under the existence of symmetries of a Lie algebroid of a cosymplectic groupoid, it is interesting to study how integration is related to reduction. Motivated by the results in \cite{BC2}, we study the reductions of the IM-forms associated with the cosymplectic forms.
    		
 We state the geometric setup for IMcs-algebroids as follows.
\begin{itemize}
\item[i)]   Let $(A,\mu,\nu,\zeta)$ be an IMcs algebroid over a Poisson manifold $(M,\pi_M)$ (cf. Definition \ref{def:IMCS});
\item[ii)]  a Lie group $\Gp$ with an action $\varphi^A:\Gp\times A\to A$ by Lie algebroid morphisms, covering an action on $M$ (denoted by $\varphi^M$);
\item[iii)] $\mu$ and $\zeta$ are $\Gp$-equivariant and $\nu$ is $\Gp$-invariant, i.e.	
    		$$\mu\circ \varphi_g^A=(d\varphi_{g^{-1}}^M)^*\circ \mu,\qquad \zeta\circ \varphi_g^M=\varphi_g^A\circ \zeta\qquad \nu \circ \varphi_g^A=\nu,\qquad  \text{ for all } g\in\Gp.$$
\end{itemize} 
Recall the map $J:A\to \mathfrak{g}^*$ from Proposition \ref{prop:over-mm}, and assume that $J^{-1}(0)$ is a vector subbundle of $A$  (for instance, if $\mu(A)$ and $\mu(A)\cap J_{can}^{-1}(0)$ have constant rank). Analogously to the construction of $\mu_{red}$ in Proposition \ref{prop:over-0-level}, we obtain $\nu_0: J^{-1}(0)\to \mathbb{R}$ as the restriction via the inclusion map $J^{-1}(0)\to A$. Moreover, by the geometric setup, $\nu_0$ is $\Gp$-invariant. Hence, it is possible to define the bundle map $\nu_{red}:A_{red}\to \mathbb{R}$ such that
\begin{equation}\label{eq:mu-red}
p^*(\nu_{red}(a))=\nu_0(p^*(a)), \qquad \text{ for all } a\in \Gamma(A_{red}).
\end{equation}
 The following diagram illustrates the pullback map that appears in this equation,
    	$$
    	\xymatrix@R=10pt@C=10pt{
    		J^{-1}(0)\simeq p^*A_{red} \ar[rr]\ar[dd]_{} \ar[rd]_{\nu_0} & &  A_{red}\ar'[d]^{}[dd] \ar[rd]^{\nu_{red}}  \\
    	& \R \ar[rr]_{}   & & \R  \\
    	M \ar[rr]_{p}  & & M/G.   \\ 
    }
    	$$	
\begin{lemma}\label{lem:IM-red}
The forms $\mu_{red}$ and $\nu_{red}$ are IM-forms on $A_{red}$, and there exists a section $\zeta_{red}\in \Gamma(A_{red})$ such that $\Ker(\mu_{red})=\langle \zeta_{red} \rangle.$
\end{lemma}
\begin{proof} 
The claim about $\mu_{red}$ has already been proved in Proposition \ref{prop:over-0-level}, and the argument is analogous for $\nu$. For the reader’s convenience, we present the proof of \eqref{eq:nu} for $\nu_{red}$. Let $a,b\in\Gamma(A_{red})$. Then
\begin{align*}
	p^*\nu_{red}([a,b])&=\nu_0([p^*a,p^*b])=\rho(p ^*a)\nu_0(p^*b)-\rho(p^*b)\nu_0(p^*a)\\
    			&=\rho(p ^*a)p^*\nu_{red}(b)-\rho(p^*b)p^*\nu_{red}(a)=p^*(\rho_{red}(a)\nu_{red}(b)-\rho_{red}(n)\nu_{red}(a)).
    			\end{align*}
Finally, note that $\zeta\in \Gamma(A)$ belongs to $\Gamma(J^{-1}(0))$, since $\mu(\zeta)=0$ and $J=J_{can}\circ \mu$. Together with the $G$-equivariance of $\zeta$, this yields $\zeta_{red}\in \Gamma(A_{red})$ via the quotient map on sections $p^*:\Gamma(A_{red})\to \Gamma(J^{-1}(0))$. It remains to show that $\Ker (\mu_{red})=\langle \zeta_{red} \rangle$. Indeed, let $a_r\in \Ker (\mu_{red})$. Then 
\[
0=p^*(\mu_{red}(a_r))=\mu_0(p^*a_r)=\mu(a),
\] 
where the second identity follows from the pullback bundle map $p^*$ and the last equality holds since $a\in \Gamma(J^{-1}(0))$. As $\Ker (\mu)=\langle \zeta \rangle$ and $\zeta\in \Gamma(J^{-1}(0))$, this implies $a=\zeta$, and therefore its quotient section is $a_r=\zeta_{red}$.
\end{proof}
We now assume that the IMcs-algebroid $(A,\mu,\nu,\zeta)$ (over a Poisson manifold $(M,\pi_M)$) is integrated by a source-simply-connected cosymplectic Lie groupoid $(\gpd,\omega,\eta,\xi)\rightrightarrows M$. 
Also consider the $\Gp$-equivariant Lie groupoid morphism
    		\begin{equation}\label{eq:Gpd-mm}
    		\Jmm:\gpd\to \g^*,
    		\end{equation}
that satisfies $\imath_{v_\gpd} \w=-d\langle \Jmm,v\rangle$ for all $v\in \g$ (see Proposition \ref{prop:over-integration}).  Since $\gpd$ is a cosymplectic groupoid, the map $\Jmm$ becomes a cosymplectic momentum map, as shown below.

\begin{lemma}
\label{lemma:IntegracionJcosymplecticomm}
The action of  $\Gp$ on $\gpd$ is by cosymplectic diffeomorphisms. Moreover, the map $-\Jmm:\gpd\to \g^*$ defined in \eqref{eq:Gpd-mm} is a cosymplectic moment map, and $\xi \langle \Jmm,v\rangle=0$ for all $v\in \g$.
\end{lemma}
\begin{proof}
To verify that $\w$ is $\Gp$-invariant, we use that that $\mu$ is $\Gp$-equivariant. Note that
\begin{align*}
	(\varphi_g^\gpd)^*\w (u,d\epsilon X)&=\w (d\varphi_g^\gpd(u),d\varphi_g^\gpd(d\epsilon X))=\w (d\varphi_g^\gpd(u),d\epsilon(d\varphi_g^M X))\\
	&=(d\varphi_{g^{-1}}^M )^*\mu(u)(d\varphi_g^M X)=\mu(u)(X)=\w(u,d\epsilon X),
\end{align*}
where $u\in A$ and $X\in TM$. Thus, from \cite[Thm. 4.6]{BC}  we conclude that $(\varphi_g^\gpd)^*\w=\w$. Similarly, $\Gp$-invariance of $\eta$ is proved by using that $\nu$ is $\Gp$-invariant. 
To prove that $-\Jmm$ is a cosymplectic moment map we  use that	the $\Gp$-action on $\gpd$ covers the action on $M$. In particular,  we have $v_\gpd(x)=v_M(x)\in T_xM$ for $v\in \g$ and  $x\in M$. Because $\eta$ is multiplicative, then $\epsilon^*\eta=0$. Thus,  $T M \subset \ker \eta$ and $(\imath_{v_\gpd}\eta)|_M=0$. Since $\Gp$ preserves the form $\eta$, we deduce $0=\mathcal{L}_{v_\gpd}\eta=d\imath_{v_\gpd}\eta$. Therefore, $\imath_{v_\gpd}\eta=0.$ Furthermore, from $\imath_{v_\gpd} \w=-d\langle \Jmm,v\rangle$,	
	we conclude $-\xi \langle \Jmm,v\rangle=\imath_\xi \imath_{v_\gpd}\w =0$.
\end{proof}

As a first application of Theorem \ref{thm: MainTheorem}, and applying Proposition \ref{prop:over-integration} together with the preceding lemmas, we obtain the main result of this section.

\begin{theorem}\label{thm:red vs int}
Let $(A,\mu,\nu,\zeta)$ be an IMcs algebroid over a Poisson manifold $(M,\pi_M)$, and let $\Gp$ be a Lie group acting freely and properly on $M$. Assume that the lifted action of $\Gp$ on $A$ preserves the IMcs-algebroid structure, as described in the geometric setup $i)$–$iii)$. Consider the map $J:A\to\mathfrak{g}^*$ from Proposition \ref{prop:over-mm}. If $J^{-1}(0)$ is a vector subbundle of $A$, then:
\begin{enumerate}
\item $(A_{red},\mu_{red},\nu_{red},\zeta_{red})$ is an IMcs-algebroid over $(M/\Gp,\pi_{red})$.
\item If $(\gpd,\omega,\eta)\toto M$ is the source simply-connected cosymplectic Lie groupoid integrating $A$, then $(\gpd_{red},\omega_{red},\eta_{red})\toto M/\Gp$ is a cosymplectic Lie groupoid integrating $(A_{red},\mu_{red},\nu_{red},\zeta_{red})$.
\end{enumerate}
\end{theorem} 

\begin{proof}
For the first item, by Lemma \ref{lem:IM-red}, it suffices to check that 
$\mu_{red}:A_{red}\to T^*(M/\Gp)$ is surjective. Take $\alpha_r\in T^*(M/\Gp)$ 
and $a\in A$ such that $\mu(a)=p^*\alpha_r$. Then, for any $u\in\mathfrak{g}$,
\[
J(a)(u)=J_{can}(p^*\alpha_r)(u)=i_{u_M}p^*\alpha_r=0.
\]
Hence, for every $\alpha_r\in T^*(M/\Gp)$ there exists $a\in J^{-1}(0)$ with 
$\mu_0(a)=\alpha_r$. Passing to the quotient of sections, we obtain 
$a_{red}$ with $\mu_{red}(a_{red})=\alpha_r$. Thus, $\mu_{red}$ is a surjective 
bundle map, and yields the central extension
\[
0\to \langle \zeta_{red} \rangle \to A_{red}
\overset{\mu_{red}}{\longrightarrow} T^*(M/\Gp)\to 0,
\]
with $\underline{\nu}_{red}(a_{red})=\nu_{red}(a_{red})\zeta_{red}$ as a splitting.

For the second item, by Lie's second theorem we integrate the $\Gp$-action on $A$ to a $\Gp$-action on $\gpd$, and we also integrate the map  $J:A\to \g^*$ to a groupoid morphism $\Jmm:\gpd\to \g^*$,  analogously to Proposition 
\ref{prop:over-integration}. Lemma \ref{lemma:IntegracionJcosymplecticomm} states that the map $\Jmm$ is also a cosymplectic moment map. By using Theorem \ref{thm: MainTheorem}, we conclude that $(\gpd_{red},\omega_{red},\eta_{red})\toto M/G$ is a cosymplectic Lie groupoid. As mentioned in Proposition \ref{prop:over-integration}, $A_{red}=Lie(\gpd_{red})$. Furthermore, let $u \in A_{red}$, then using the groupoid morphism $(P,p):(\Jmm^{-1}(0)\to M)\to (\gpd_{red}\to M/G)$, we have
    			\begin{align*}
    				p^*\epsilon^*(\imath_{u^R}\eta_{red})&=\epsilon^*P^*(\imath_{u^R}\eta_{red})=\epsilon ^*(\imath_{p^*(u^R)}P ^*\eta_{red})=\epsilon ^*(\imath_{p^*(u^R)}i^*\eta)\\ &=\nu_0(p^*(u))=p^*(\nu_{red}(u)). 
    			\end{align*} 	
    			Since $p$ is a surjective submersion we get $\epsilon^*(\imath_{u^R}\eta_{red})=\nu_{red}(u)$. Analogously for $\mu_{red}$. 
\end{proof}
 		
        In other words, for a cosymplectic groupoid, the reduction of its infinitesimal counterpart is the same as the infinitesimal version of its cosymplectic reduction. Thus, we have that		
    		$$\xymatrix{
    			(\gpd, \omega, \eta) \ar[d]^{\text{Lie}}  \ar[r]^{\hspace{-10mm}\text{red}}& (\gpd_{red}, \omega_{red}, \eta_{red}) \ar[d]^{\text{Lie}}\\
    			(A,\mu,\nu,\zeta)\ar[r]^{\hspace{-10mm}\text{red}} & (A_{red}, \mu_{red},\nu_{red},\zeta_{red}),
    		}$$
    		commutes. The following is an immediate consequence of this theorem.

    		\begin{corollary}
    		\label{coro:main-coro}
Let $(T^*M\oplus \R , \mu=pr_1,\nu=pr_2,\zeta=\partial r)$ be a trivial central extension Lie algebroid with brackets as in \eqref{prop:centralext A}. If there is a $\Gp$ action on $M$ which lifts to an action by Lie algebroid morphism on $A$ so that  $\mu$ is equivariant  and  $\nu$ is invariant, then 	
$$(T^*(M/\Gp)\oplus \R , \mu_{red}=pr_1,\nu_{red}=pr_2,\partial r)$$
is also a central extension algebroid isomorphic to $A_{red}$.
    		\end{corollary}

	\subsection{Reduction of the symplectic leaf}
 	 	\label{sect:consecuencias2}
 	 	A natural question that arises is the symplectic reduction of the leaves of a cosymplectic groupoid. In this section we study how symplectic and cosymplectic reduction are related in the restriction to the symplectic leaf containing the units of a cosymplectic groupoid. 
 
		Given a cosymplectic groupoid $(\gpd,\omega,\eta)\toto M$,  let $\Sigma$ be the symplectic leaf of $\pi_\gpd$ containing $M$. Let $L:\Sigma\to \gpd$  denote the inclusion. In \cite{FI} it is proved that $\Sigma\subset \gpd$ is a Lie subgroupoid, and  $(\Sigma,L^*\omega)\toto M$ is  a symplectic groupoid integrating $\pi_M$. Assuming that we impose symmetries on the cosymplectic groupoid $(\gpd,\w,\eta)$ so that we can construct its cosymplectic reduced groupoid, we inquire about the symmetries and the reduction of $(\Sigma,L^*\w)\toto M$. 
		We want to restrict the action $\Gp$ to the leaf $\Sigma$, and we want to define a symplectic moment map.
		
	To show that the action of $\Gp$ on $\gpd$ is restricted to an action on $\Sigma$, we first state a basic fact about morphisms between Poisson manifolds.
\begin{lemma}
	\label{lem:PoissonMorLeaves}
Let $M$ and $N$ be Poisson manifold, and let $\varphi:M\to N$ be a Poisson diffeomorphism. Then, $\varphi$ preserves the symplectic foliation.
\end{lemma}
\begin{proof}
    This fact follows from the $\varphi$-related Hamiltonian vector fields on the respective Poisson manifolds.
\end{proof}

 It is important to highlight that the leaves of the symplectic foliation may be non-invariant by the action of a group by Poisson automorphisms (cf. \cite[Ex. 4.2.5]{ORbook}). Moreover, in the case of \textit{standard momentum maps} on Poisson manifolds, a necessary condition for the existence of a momentum map is that the action by Poisson morphisms being leaf preserving, which does not happen in general \cite[\S 4.5]{ORbook}. However, in the case of a cosymplectic groupoid, the action does preserve the symplectic leaf containing the units. Thus,  we can restrict the action to this leaf.
\begin{lemma}
	\label{lem:resrtictaction}
	Let $\Gp$ be a Lie group acting on a cosymplectic groupoid $(\gpd, \omega, \eta)\toto M$ by cosymplectic groupoid automorphisms. Let $\Sigma$ be the symplectic leaf of $\pi_\gpd$ containing $M$. Then, we can  restrict the action $\Gp\action \gpd$ to a $G$-action on $\Sigma$, preserving  the form $L^*\w$.
\end{lemma}

\begin{proof}
The action $\Gp\action \gpd$ is by Poisson diffeomorphisms. By Lemma \ref{lem:PoissonMorLeaves}, any left multiplication
either transforms  the  leaf $\Sigma$ into another leaf of the symplectic foliation, or leaves it invariant. Furthermore, for any $g\in G$, the units $M$ are contained in $\varphi_g(\Sigma)$. Therefore, by the definition of $\Sigma$, we have $\varphi^\gpd_g(\Sigma)=\Sigma$. Let $\varphi^\Sigma$ denote the action of $\Gp$ on $\Sigma$.    In addition, we verify that
$$(\varphi^\Sigma_g)^*L^*\w=L^*(\varphi_g^\gpd)^* \w=L^*\w.$$
\end{proof}
With this action of $\Gp$ on $\Sigma$ by symplectic groupoid automorphisms, we can show that the groupoid morphism $\Jmm_0:\Sigma\xto{L}\gpd\xto{\Jmm}\g^*$ is a symplectic moment map. For the convenience of the reader, we recall the assumptions of Theorem \ref{thm: MainTheorem}: Let $\Gp$ be a Lie group acting on a cosymplectic groupoid $(\gpd, \omega, \eta)\toto M$ by cosymplectic groupoid automorphisms.	Let $\Jmm:\gpd\to \g^*$ be a cosymplectic moment map such that it is a Lie groupoid morphism, and $\xi(\langle \Jmm,v\rangle)=0$ for all $v\in \g$. Moreover, assume that $\Gp$ acts freely and properly on $M$.

\begin{lemma}
	\label{prop:MomentMapaSymplectic}
	Let $\Gp$ be a Lie group acting on a cosymplectic groupoid $(\gpd, \omega, \eta)\toto M$,  and  let $\Jmm:\gpd\to \g^*$ be a cosymplectic moment map satisfying the conditions of Theorem \ref{thm: MainTheorem}. Let $\Sigma$ be the symplectic leaf containing $M$.  Then the groupoid morphism $\Jmm_0:(\Sigma, L^*\omega)\to \g^*$ defined by $L^*\Jmm$ is a symplectic moment map with respect to the $\Gp$-action in Lemma \ref{lem:resrtictaction}.
\end{lemma}
 \begin{proof}
 Since the action of $\Gp$ on $\Sigma$ is the restriction of the action $\varphi^\gpd:\Gp\action \gpd$, for any $g\in \Gp$  $L\circ\varphi^\Sigma_g=\varphi^\gpd_g\circ L$. Thus, the $\Gp$-equivariance of $\Jmm_0$ follows from the $\Gp$-equivariance of $\Jmm$, that is $Ad_g^*\Jmm_0=\Jmm\circ \varphi^\gpd_g\circ L=\Jmm_0\varphi^\Sigma_g.$

Moreover, for $v\in \g$, the restriction $v_{\gpd}|_\Sigma$ of its fundamental vector field  is actually a vector field on $\Sigma$. Hence, $L^*\imath_{v_\gpd}\omega=\imath_{v_\gpd} L^*\omega$. Since $\Jmm$ is a cosymplectic moment map and $\xi(\langle J,v\rangle)=0$, we conclude $$d\langle \Jmm_0, v\rangle=L^*d\langle \Jmm, v\rangle=L^*(\imath_{v_\gpd}\omega+\xi\langle \Jmm, v\rangle\eta)=L^*\imath_{v_{\gpd}}\omega=\imath_{v_\gpd}L^*\omega.$$
 \end{proof}

Under the above assumptions, the conditions of \textit{cosymplectic Hamiltonian action} and \textit{symplectic Hamiltonian action} are satisfied.  Therefore, we can perform the structure reductions. These two reductions can be done simultaneously by observing that the diagram 
 	 	 		 \begin{equation}
 	 	 		 \label{eq:SympCosymdiagram}
 	 	 		 \xymatrix@R=10pt@C=10pt{
 	 	 		 	\Jmm_0^{-1}(0) \ar[rr]^{L}\ar[dd]_{L_0} \ar[rd]^{P_0} & & \Jmm^{-1}(0) \ar'[d]^{i}[dd] \ar[rd]^{P}  \\
 	 	 		 	& \Jmm_0^{-1}(0)/\Gp \ar[rr]_{L_{red}\hspace{5mm}}   & & \gpd_{red}  \\
 	 	 		 	(\Sigma, L^*\omega) \ar[rr]_{L}  & & (\gpd, \omega,\eta).   \\ 
 	 	 		 }
 	 	 		 \end{equation}
 	 commutes. The reduced symplectic groupoid $\Jmm_{0}^{-1}(0)/G$ is constructed by considering the symplectic structure as in \cite{MW}, which is compatible with the structural maps of the groupoid \cite[Prop. 4.6]{FOR}.  As a direct consequence we get the next theorem.
 	 	\begin{theorem}
 	 		\label{thm:mainSimplecticLeaf}
 	 		Let $\Gp$ be a Lie group acting on a cosymplectic groupoid $(\gpd, \omega, \eta)\toto M$ by cosymplectic groupoid maps,  and  let $\Jmm:\gpd\to \g^*$ be  a cosymplectic moment map satisfying the conditions of Theorem \ref{thm: MainTheorem}. Let $\Sigma$ the symplectic leaf containing $M$. Then 
 	 		$$L_{red}^*\omega_{red}=(L^*\omega)_{red},$$
 	 		where $L:\Sigma\to \gpd$ is the inclusion and $L_{red}: \Jmm_0^{-1}(0)/\Gp\to \Jmm^{-1}(0)/\Gp$.
 	 	\end{theorem}
 	 	\begin{proof}
	 		Since $\Gp$ acts freely and properly on $M$, the $\Gp$-action on $\Sigma$ is also free and proper \cite{FOR}. Thus, by a similar argument as in Remark \ref{rmk:0-regular}, 0 is a regular value of $\Jmm_0$. Moreover,  $(\Jmm_0^{-1}(0)/\Gp, (L^*\w)_{red})\toto M/G$ is a symplectic groupoid \cite{FOR}.
	 		The reduced symplectic form $(L^*\omega)_{red}$
            on $\Jmm_0^{-1}(0)/\Gp$ is characterized by $$L_0^*(L^*\omega)=P_0^*(L^*\omega)_{red},$$ where $L_0:\Jmm_0^{-1}(0)\to \Sigma$ is the inclusion map and $P_0:\Jmm_0^{-1}(0)\to \Jmm_0^{-1}(0)/\Gp$ is the quotient map. By the commutativity in \eqref{eq:SympCosymdiagram}, we have 
 	 		$$L_0^*L^*\omega=L^*i^*\omega=L^*P^*\omega_{red}=P_0^*L_{red}^*\omega_{red}.$$
 	 		Therefore, $P_0^*(L^*\omega)_{red}=P_0^*L_{red}^*\omega_{red}$. As usual, $P_0$ is a surjective submersion, so the desired result is obtained.
 	 	\end{proof}

 	 	\subsection{Reduction of the multiplicative Chern class}
 	 	 \label{sect:consecuencias3} 
 
As mentioned in Section \ref{sec:cs-gpd}, \cite{FI} establishes a relation between $S^1$-central extensions with vanishing multiplicative Chern class and cosymplectic groupoids. We use this identification and Theorem \ref{thm: MainTheorem} to show that if the multiplicative Chern class vanishes, the same happens in the quotient.

Let $(\gpd,\omega)$ be a corank 1, orientable, proper oversymplectic groupoid over $M$, whose  foliation $\ker \omega$ is simple. As in Theorem \ref{thm:CC-cosym}, there is a one-to-one correspondence between a cosymplectic groupoid structure on $(\gpd,\omega)$ and the vanishing of the multiplicative Chern class. Thus, if $\gpd$ is a cosymplectic groupoid, we are interested in understanding the behavior of the multiplicative Chern class under the reduction procedure described in Theorem \ref{sect: MainTheorem}. Assume a $G$-action on $\gpd$ and a cosymplectic moment map $\Jmm:\gpd\to \g^*$ as in Theorem \ref{thm: MainTheorem}. Note that although $\gpd$ is proper, the reduced groupoid $\gpd_{red}\toto M/G$ may be a non-proper groupoid. However, with additional conditions the reduced groupoid is proper.
 	 	\begin{lemma}
 	 		\label{lem:PropeGroupoid}
 	 		Let $\Gp$ be a Lie group acting on a proper Lie groupoid $\gpd\toto M$ by groupoid morphisms. Let $\mathcal{H}\toto M$ be a closed Lie subgroupoid of $\gpd$. Suppose that $\mathcal{H}$ is closed under the action of $\Gp$. If  $\Gp$ acts  freely and properly on $M$, and the map $M\xto{p} M/\Gp$ is proper, then $\mathcal{H}/\Gp\toto M/\Gp$ is a proper Lie groupoid.
 	 	\end{lemma}
 	 	\begin{proof}
Since the action of $\Gp$ on $\mathcal{H}$ is free and proper \cite{FOR}, the quotients $\mathcal{H}/\Gp$ and $M/\Gp$ are smooth manifolds. As the structural maps of $\mathcal{H}\toto M$ are $\Gp$-equivariant, then they descend to $\mathcal{H}/\Gp$. Now, we show that under these hypotheses $\mathcal{H}$ is a proper groupoid. To show that  $(s,t):\mathcal{H}\to M\times M$ is a proper map, we consider a sequence of arrows $y_n\xfrom{g_n}x_n\in \mathcal{H}$ such that $(x_n, y_n)\to (x,y)\in M\times M$, and we show that $g_n$ admits a convergent subsequence in $\mathcal{H}$ (see \cite{dH,Bour}). Since $\gpd$ is a proper groupoid, there exists a subsequence $g_{n_k}\to g\in \gpd$. Because $\mathcal{H}\subset \gpd$ is closed, then $g\in \mathcal{H}$.

Let $[y_n]\xfrom{[g_n]}[x_n]\in \mathcal{H}/\Gp$ be such that $([x_n],[y_n])\to ([x],[y])\in M/\Gp\times M/\Gp$. We want to show that $[g_n]$ admits a convergent subsequence in $\mathcal{H}/\Gp$. Since $(p,p):M\times M\to M/\Gp\times M/\Gp$ and $(s,t):\mathcal{H}\to M\times M$ are proper maps, there exists a subsequence $y_n\xfrom{g_{n_k}}x_n$ converging to $g\in \mathcal{H}$. By the continuity of the map $\mathcal{H}\to \mathcal{H}/\Gp$, we conclude that $[g_{n_k}]\to [g]\in \mathcal{H}/\Gp$.
 \end{proof}
 
 Next, we state a result which follows directly from Theorem \ref{thm: MainTheorem} and the characterization of cosymplectic groupoids in terms of $S^1$-central extensions with vanishing multiplicative Chern class presented in Theorem \ref{thm:CC-cosym}.

	\begin{proposition}
		\label{thm:mainChern}
		Let $\Gp$ be a Lie group acting on a cosymplectic groupoid $(\gpd, \omega, \eta)\toto M$,  and  let $\Jmm:\gpd\to \g^*$  be a cosymplectic moment map satisfying the conditions of Theorem \ref{thm: MainTheorem}. If 
		$p:M\to M/G$ is a proper map, then for $(\gpd_{red},\w_{red},\eta_{red})\toto M/\Gp$ the corresponding $S^1$-central extension has vanishing multiplicative Chern class.
	\end{proposition}
\begin{proof}
	By using Theorem \ref{thm: MainTheorem} we obtain a cosymplectic groupoid $(\gpd_{red},\w_{red},\eta_{red})\toto M/G$.	Since $M\to M/G$ was assumed to be a proper map, by Lemma \ref{lem:PropeGroupoid}, $\gpd_{red}$ is proper. Consequently, by using  Theorem \ref{thm:CC-cosym}, we conclude that the corresponding $S^1$-central extension of  $(\gpd_{red},\w_{red},\eta_{red})\toto M/\Gp$  has vanishing multiplicative Chern class.
\end{proof}

If we start with a corank 1, orientable, proper oversymplectic groupoid, with simple foliation $\ker \w$, such that the corresponding $S^1$-central extension has vanishing multiplicative Chern class,  the following theorem gives a way to produce new groupoids where the multiplicative Chern class vanishes.

\begin{theorem}
\label{thm:mainChern0}
Let $(\gpd, \omega)\toto M$ be a corank 1, orientable, proper oversymplectic groupoid, with simple foliation $\ker \w$, such that $H^1(\gpd)=0$. Let $\Gp$ be a compact, connected Lie group acting on $\gpd$ by Lie groupoid automorphisms preserving $\omega$. The following is true:
\begin{enumerate}
	\item If   the corresponding $S^1$-central extension has vanishing multiplicative Chern class, then there exists a cosymplectic groupoid $(\gpd,\w,\bar{\eta})$ and a moment map $\Jmm:\gpd\to \g^*$, such that $\xi(\langle \Jmm,v\rangle)=0$ for all $v\in \g$.
	\item If, in addition, $\Jmm$ is a groupoid morphism, $G$ acts freely on $M$, and $M\to M/G$ is a closed map, then for $(\gpd_{red},\w_{red},\bar{\eta}_{red})\toto M/\Gp$ the corresponding $S^1$-central extension has vanishing multiplicative Chern class.
\end{enumerate}
\end{theorem}

\begin{proof} 
To prove item (1), we find an $\Gp$-invariant multiplicative 1-form $\bar \eta$ on $\gpd$, such that $(\gpd,\omega,\bar \eta)$ is a cosymplectic groupoid. Namely, the condition on the multiplicative Chern class defines a cosymplectic structure $(\gpd,\omega,\eta)$ (cf. Theorem \ref{thm:CC-cosym}). Because $\Gp$ is compact,  we consider the average
$$\bar{\eta}=\int_\Gp(\varphi^\gpd_g)^*\eta d\Lambda_g,$$
where $d\Lambda$ is the Haar measure of the compact group $\Gp$. From Haar measure invariance it follows that $\bar{\eta}$ is a $\Gp$-invariant, closed,  multiplicative 1-form in $\gpd$. Let $\xi$ be the Reeb vector field  of $(\gpd,\omega,\eta)$. Note that $d\varphi^\gpd_g\xi=f_g\xi$ for some non-vanishing $f_g\in C^{\infty}(\gpd)$, thus $\imath_{\xi}\bar \eta=\int_Gf_gd\Lambda_g$. Furthermore, if $f_g(h)=0$ for some $h\in \gpd$, because $\varphi^\gpd_g$ is a diffeomorphism, we conclude that $\xi(h)=0$, which  cannot happens.  Consequently, by the connectedness of $G$, $f_g$
has the same sign for all $g\in G$.

 We claim that  $(\gpd, \omega, \bar \eta)$ is a cosymplectic manifold. In fact,  $\xi$ satisfies $\imath_\xi\omega=0$ and $\imath_\xi\bar \eta\neq 0$. Thus, if the map $\flat:\mathfrak{X}(\gpd)\to \Omega^1(\gpd)$ defined in \eqref{eq:flat} associated with $\bar \eta$ is not injective, then we can consider $X\neq 0$ in the kernel of $\flat$. Note that $0=\imath_\xi\flat(X)=(\imath_X\bar \eta)(\imath_\xi \bar \eta)$. Hence,  $\bar \eta$ vanishes on $X$, and $\imath_X\omega=\flat(X)=0$. Thus, $X=f \xi$ for some non-vanishing smooth function $f$. However, $\imath_X \bar \eta=f\imath_\xi\bar \eta$, which is a contradiction. Consequently, $\flat$ is an isomorphism. We conclude this claim, by considering the normalization $\bar \xi$ of $\xi$, such that $\imath_{\bar \xi}\bar \eta=1$, and Remark \ref{rmk:flat}.
 
 Since $\bar{\eta}$ is invariant by the action, we get that $d(\imath_{v_{\gpd}}\bar{\eta})=\mathcal{L}_{v_{\gpd}}\bar{\eta}=0$. Because $\epsilon^*\bar{\eta}=0$, then $0=\imath_{v_M}\epsilon^*\bar{\eta}|_x=\imath _{v_{\gpd}}\bar{\eta}|_{\epsilon(x)}$ for $x\in M$.  Hence, $\imath_{v_{\gpd}}\bar{\eta}=0$ for all $v\in \mathfrak{g}$. The existence of the cosymplectic moment map $\Jmm:\gpd\to \g^*$ follows from Proposition \ref{prop:existence of J}.
 
 To prove item (2),  we consider the reduced cosymplectic groupoid $(\gpd_{red},\omega_{red}, \bar \eta_{red})\toto M/G$ obtained from Theorem \ref{thm: MainTheorem}. In Lemma \ref{lem:PropeGroupoid}, considering that the group $\Gp$ is compact, the condition that the function $M\to M/\Gp$ is proper is replaced by that of being a closed map, therefore $\gpd_{red}$ is proper. From Theorem \ref{thm:CC-cosym}, we conclude that the corresponding $S^1$-central extension of  $(\gpd_{red},\w_{red},\bar \eta_{red})\toto M/\Gp$  has vanishing multiplicative Chern class. 
 \end{proof}
	
 	\section{Further direction of study}\label{sec:FD}
 \begin{enumerate}
 	\item  In the cosymplectic reduction Theorem \ref{thm:cosymp reduction} it is assumed that  $v_Q$ is the Hamiltonian vector field of $\langle J,v\rangle$ and  $\xi\langle J,v\rangle=0$, this in particular implies that $\mathcal{L}_{v_Q}\w=0 $ and $\mathcal{L}_{v_Q}\eta=0$. However, we can suppose conformal symmetries, which leads us to consider the conditions
	$$\mathcal{L}_{v_Q}\w=-d(\xi\langle J,u\rangle)\wedge \eta \text{\ and\ }\mathcal{L}_{v_Q}\eta=f_v\eta$$ 
	for some non-vanishing function $f_v$ for each $v\in \g$.	In this context, the reduction theorems are different, and would state others relation with respect to the cosymplectic groupoids reduction.   More generally, we may consider two broader interpretations of conformal symmetries in cosymplectic structures, as we present below.
    \begin{itemize}
        \item A cosymplectic structure can be regarded as a symplectic structure on a line bundle $L\to M$ (cf. \cite{TVY}). 
It is given by a 2-form $\omega \in \Omega^2(DL,\mathbb{R}_M)$, where $DL$ denotes the Atiyah algebroid and $\mathbb{R}_M$ is the trivial line bundle. 
In this setting, a symplectic structure on $M$ with values in $L$ is a pair $(L,\omega)$, where $\omega$ is non-degenerate and closed with respect to the trivial representation of $DL$ on $\mathbb{R}_M$.

In this context, one expects a cosymplectic groupoid to be a line VB-groupoid 
$$
\xymatrix{
	L\ar@<0.5ex>[r] \ar@<-0.5ex>[r]\ar[d]^{} &L_M\ar[d]\\  
	\gpd\ar@<0.5ex>[r] \ar@<-0.5ex>[r]&M ,
}
$$
together with a multiplicative cosymplectic structure $\omega\in \Omega^2(DL,\mathbb{R}_\gpd)$. 
The approach to conformal symmetries in this framework is motivated by the trivial line bundle $L=\mathbb{R}\times M$, 
where a cosymplectic isomorphism with values in $L\to M$ can be translated into a diffeomorphism $\phi:M\to M$ together with a smooth function $g$ such that
$$
\omega=\phi^*\omega+dg\wedge \phi^*\eta,\qquad \eta=\phi^*\eta .
$$ 
In this way, the action of a Lie group leads to
$$
L_{u_M}\omega = dG_u\wedge \eta,\qquad L_{u_M}\eta = 0,
$$
where the functions $G_u$ provide a natural candidate for a moment map of the action.

\item A cosymplectic structure can be regarded as a closed 2-form that restricts to a symplectic form on an involutive corank-one distribution $H\subset TM$.\footnote{Contact geometry appears as the non-integrable case.} 
Multiplicativity should then be interpreted as a condition on the distribution $H\subset T\gpd$, requiring it to define an involutive wide subgroupoid (cf.~\cite{CSa}). 
A cosymplectic isomorphism should preserve the distribution; that is, it consists of a diffeomorphism $\phi$ together with a smooth function $g$ such that 
\[
\omega = \phi^*\omega, 
\qquad \text{and} \qquad 
\phi^*\eta = g\,\eta,
\]
since we must have $\phi_*H \subseteq H$. 
For a Lie group action, this translates into
\[
L_{u_M}\omega = 0, 
\qquad 
L_{u_M}\eta = G_u \,\eta.
\]
\end{itemize}

	
	\item  In \cite[Prop. 3.3]{FI}, given an IMcs-algebroid $(A,\mu,\nu,\zeta)$, the authors establish an isomorphism of Lie bialgebroids $(A,A^*) \simeq (T^*M \oplus \mathbb{R}, TM \oplus \mathbb{R}).$ However it is also known 
	the existence of generalized Lie bialgebroids \cite{IM} and triangular generalized Lie bialgebroids
	as in \cite{IMLP}. The reduction results of this article can be used to reduce such bialgebroids. They can also be used to compare other geometric terms, such as the modular class in the reduction in \cite{IMLP}.
	
	 \item Integration of Poisson structures can also be understand via Poisson sigma models \cite{CF}. 
	If the groupoid $\gpd$ is cosymplectic, it has a Poisson bracket. Furthermore,  the units $M$ also has a Poisson structure and a Lie algebroid  structure. Results in \cite{FI} and in these notes can be extended to the integration theory given in \cite{CF} via sigma models.
	
	\item The construction of the moment map in Section \ref{sect:consecuencias1} comes from the ideas in \cite{BC2}, where there is an action of a Lie group on an algebroid which lifts an action on the basis. By the Lie's second theorem, this action is integrated to an action of a Lie group on a Lie grouproid. It would be valuable to see if the ideas worked out here can be generalized by considering the action of a 2-group on a cosymplectic groupoid and its infinitesimal counterpart.
\end{enumerate}
\vspace{0.1cm}

\begin{small}
\noindent \textbf{Statements and Declarations:} The authors have no conflict of interest to declare that are relevant to this article.\\
\noindent \textbf{Data Availability Statements:} Data sharing not applicable to this article as no datasets were generated or analysed during the current study.
\end{small}

\bibliographystyle{abbrv}
\bibliography{ReferencesLMP}

 \vspace{0.1cm}
 			\sf{\noindent Daniel L\'opez Garcia\\
 				Universidade de Federal Fluminense (UFF),  \\ 
 			São Domingos, Niterói, 24210-200, Rio de Janeiro, Brazil.\\
 				danielflg@id.uff.br}\\
 			    Orcid: 0000-0002-8436-2599\\
 			
 			\sf{\noindent Nicolas Martinez Alba\\
 				Universidad Nacional de Colombia (UNAL), \\ 
 				Av- Cra 30 \#45-3, Bogot\'a 16486,
 				Bogot\'a, Colombia.\\
 				nmartineza@unal.edu.co}\\
	 			Orcid: 0000-0002-2697-3228
 \end{document}